\documentclass[10pt]{article}
\font\smallit=cmti10

\usepackage{amssymb,latexsym,amsmath,epsfig,amsthm} 
\usepackage[utf8]{inputenc}
\usepackage[english]{babel}
\usepackage{amsmath,amsthm}
\usepackage{amsfonts}
\usepackage{amssymb}
\usepackage{mathtools}
\usepackage{physics} 
\usepackage{listings}
\usepackage{url}
\usepackage{hyperref}
\hypersetup{
  colorlinks   = true,    
  urlcolor     = black,    % Colour for external hyperlinks
  linkcolor    = black,    % Colour of internal links
  citecolor    = black      % Colour of citations
}
\makeatletter 

\renewcommand\section{\@startsection {section}{1}{\z@} {-30pt \@plus
    -1ex \@minus -.2ex} {2.3ex \@plus.2ex}
  {\normalfont\normalsize\bfseries\boldmath}}

\renewcommand\subsection{\@startsection{subsection}{2}{\z@}
  {-3.25ex\@plus -1ex \@minus -.2ex} {1.5ex \@plus .2ex}
  {\normalfont\normalsize\bfseries\boldmath}}

\renewcommand{\@seccntformat}[1]{\csname the#1\endcsname. }

\makeatother

\newtheorem{theorem}{Theorem}
\newtheorem{lemma}{Lemma}
\newtheorem{proposition}{Proposition}
\newtheorem{corollary}{Corollary}

\theoremstyle{definition}
\newtheorem{definition}{Definition}

\newtheorem{remark}{Remark}

\newcommand\A{\mathbb{A}}
\newcommand\NN{\mathbb{N}}
\newcommand\QQ{\mathbb{Q}}
\newcommand\ZZ{\mathbb{Z}}
\newcommand\RR{\mathbb{R}}
\newcommand\DD{\mathcal{D}}
\newcommand\hh{\mathcal{H}}
\newcommand\SC{\mathcal{S}}
\newcommand\N{\mathop{\rm N}\nolimits}
\begin{document}

\begin{center}
  \uppercase{\bf \boldmath The Dedekind-Hasse's Criterion\\[2mm] in
    Quaternion Algebras} \vskip 20pt
  {\bf Adriana Cardoso}\\
  {\smallit Centro de Matem\'atica da Universidade do Porto \\
    4169-007
    Porto, Portugal}\\
  {\tt adrianacardos11@gmail.com}\\
  \vskip 10pt
  {\bf António Machiavelo}\\
  {\smallit Centro de Matem\'atica da Universidade do Porto \\
    4169-007
    Porto, Portugal}\\
  {\tt ajmachia@fc.up.pt}\\
\end{center}
\vskip 10pt

\centerline{Preprint date: \today}
%\centerline{\smallit Received: , Revised: , Accepted: , Published: }

\vskip 30pt 

\centerline{\bf Abstract}

\noindent We show that a criterion for an integral domain to be a
principal ideal domain (PID), due to Dedekind and Hasse, can also be
applied in quaternion orders, and that it can be used to build a
finite algorithm to determine if a given order is a principal left (or
right) ideal domain. Using this algorithm, we give an alternative
proof that the maximal orders of discriminant 7 and 13, which are
non-Euclidean, are PIDs.

We also provide a completely arithmetic proof of a result of Gordon
Pall that shows that, in an order that is a PID, an element of whose
norm is divisible by an integer $m$ always has a left and a right
divisor with norm $m$. This easily yields the existence and uniqueness
(up to associates) of factorizations of a quaternion modeled on a
factorization of its norm.

\vspace{2mm}
 
\noindent\small
\textbf{Keywords:} Quaternions, Hurwitz Integers, Quaternion Algebras,
Quaternion Orders, Non Commutative Principal Ideal Domains.
\normalsize

\pagestyle{headings}
%\markright{\smalltt INTEGERS: 26 (2026) \hfill}
\thispagestyle{empty}
\baselineskip=12.875pt
\vskip 30pt 

%%%%%%%%%%%%%%%%%%%%%%%%%%%%%%%%%%%%%%%%%%%%%%%%%%%%%%%
\section{Introduction}

Inside the rational quaternions, there is a maximal order, the ring of
Hurwitz integers, $\hh$, named after Adolf Hurwitz, who studied this
ring in great detail in his book \cite{Hurwitz}. In particular,
Hurwitz proved a factorization theorem for $\hh$ that states that any
primitive quaternion has a factorization into Hurwitzian primes
modeled on a rational factorization of its norm, unique up to
unit-migration (for a modern proof, see \cite[Thm.~2,
p.~57]{Conway_Smith}).

Pall proved, in \cite[Thm.~2, p.~288]{Pall46}, using the theory of
quadratic forms, that, in certain orders of generalized quaternions,
given an element of that order and an integer that divides its norm,
then that element has a unique, up to associates, left divisor, as
well as a unique right divisor, whose norm is that integer. Here we
provide a completely arithmetic proof of this result for orders in
quaternion algebras that are principal ideal domains (PIDs). From
this, one can easily deduce the unique factorization modeled on a
factorization of the norm.

For a maximal order to be a PID it is sufficient and necessary for it
to have class number 1, i.e.~the number of isomorphism classes of
invertible (left or right) ideals is one, as all the ideals are
invertible. In 1980, in her book \cite{Vigneras}, Vignéras determined
the ten Eichler orders with class number 1, which in particular
include the maximal orders of discriminant equal to
$ 2, 3, 5, 7 \text{ or } 13$. She also remarked that the only norm
Euclidean orders were the ones with discriminant $2, 3 \text{ or }
5$. Later, in 1995, Brzezinski showed in \cite{Brzezinski} that there
were exactly 24 isomorphism classes of orders with class number 1 in
rational definite quaternion algebras, which included the ten Eichler
orders that Vignéras found. Here, we use the Dedekind-Hasse's
criterion to provide an alternative proof of the fact that the maximal
orders of discriminant $7$ and $13$ are PIDs.

In 2008, Deutsch, in \cite{Deutsch}, set to find a quaternion proof of
the universality of some quadratic forms of the form
$t^2 + a x^2 + b y^2 + c z^2$, by studying the associated quaternion
orders, the ones that are Euclidean for the norm. He succeeded in
studying the orders of discriminant $2$ and $3$, and was then followed
by Fitzgerald, in 2011, who provided a basis for the orders of
discriminant $5, 7$ and $13$, in \cite{Fitzgerald}. Together, these
two articles provide an arithmetic approach to the fact that the
orders of discriminant $2, 3$ and $5$ are Euclidean for the norm, and
thus PIDs.

We start, in section~\ref{fact}, by providing an arithmetic proof of
Pall's result mentioned above.  In section~\ref{criterion}, we show
that the Dedekind-Hasse's criterion (see \cite[Thm.~9.5,
p.~124]{Pollard_Diamond}) is also valid for quaternion orders.  In
section \ref{fin_alg}, we show how to use the Dedekind-Hasse's
criterion as the basis of a finite procedure to check if a quaternion
order is a PID, and in section~\ref{orders7and13} we apply this
algorithm to give a proof that the orders of discriminant $7$ and
$13$, although not norm Euclidean, are principal ideal
domains. Finally, in section~\ref{sec:PARI} we provide the PARI/GP
code that implements our algorithm.
%%%%%%%%%%%%%%%%%%%%%%%%%%%%%%%%%%%%%%%%%%%%%%%%%%%%%%%
%%%%%%%%%%%%%%%%%%%%%%%%%%%%%%%%%%%%%%%%%%%%%%%%%%%%%%%
\section{Factorization in a Principal Ideal Order}\label{fact}

In the following sections, let $H$ be an order in the rational
quaternion algebra $\A = \left ( \frac{-a,-b}{\QQ} \right)$, where
$a,b>0$ and such that $\A$ is a division algebra, with the $\QQ$-basis
$(1, i, j,k)$ such that $i^2=-a$ and $j^2=-b$.

\begin{definition}
  We say an element $\alpha \in H$ is \textit{primitive} if there is
  no integer $n \neq -1,1$ such that $n \mid \alpha$.
\end{definition}

Recall that a ring $R$ is said to be a \textit{left (right) principal
  ideal domain} (PID) if it is a (not necessarily commutative) domain
and every left (right) ideal is principal. The fact that we are
interest only on orders inside quaternion algebras, which are equipped
with a natural involution, implies that to every left property that
holds on the order there is an entirely analogous right property that
also holds.

\begin{theorem}\label{divisor}
  Let $H$ be a left (right) PID, $\theta \in H$ be a primitive
  quaternion, and $m \in \NN$ be such that $ m \mid \N(\theta)$. Then
  there is only one right (left) divisor of $\theta$ of norm $m$, up
  to left (right) associates.
\end{theorem}

\begin{proof}
  Let $H$ be a left PID.  Let $\gamma \in H$ be such that
  $H \theta + H m = H \gamma$, and let $\alpha, \beta \in H$ be such
  that $\gamma = \alpha \theta + \beta m$. Then
  \[\N(\gamma) = (\alpha \theta + \beta m ) (\overline{\theta}
    \overline{\alpha} + \overline{\beta}m) = \N(\alpha) \N(\theta) +
    m^2 \N(\beta) + m (\alpha \theta \overline{\beta} + \beta
    \overline{\theta} \overline{\alpha}),\] which is divisible by $m$,
  by hypothesis. Let $t \in \NN$ be such that $\N(\gamma)=mt$. Since
  $\theta, m \in H\gamma$, there are
  $f,g \in H \colon \; \theta = f \gamma$ and $m = g \gamma$. Then,
  taking the norm, we have $m^2 = \N(g) \N(\gamma) = \N(g) mt$, and
  thus $m = \N(g)t$. From this it follows that
  \[g \gamma = m = \N(g) t = g\overline{g}t,\] and so
  $\gamma = \overline{g}t$. But then
  $\theta = f \gamma = f \overline{g}t$, which implies that
  $t \mid \theta$, and thus we must have $t = 1$, since $\theta$ is
  primitive. This shows that $\N(\gamma) = m$, and since
  $\theta = f \gamma$, we have shown that $\theta$ does indeed have a
  right divisor norm $m$.
	
  Now, if $\delta$ is another right divisor of $\theta$ of norm $m$,
  then $\theta = \epsilon \delta$ for some $\epsilon \in H$.  But then
  \[\gamma = \alpha \theta + \beta m = \alpha \epsilon \delta + \beta
    \overline{\delta} \delta = (\alpha \epsilon + \beta
    \overline{\delta}) \delta.\] The fact that
  $\N(\gamma) = \N(\delta)$ then entails that $\delta$ and $\gamma$
  are left associates.
\end{proof}

Recall that, we say an element $\pi \in H$ is a \textit{prime
  quaternion} (in $H$) if $\pi = \alpha \beta $ implies that either
$\alpha$ or $\beta$ is a unit, that is, if and only if the norm of
$\pi$ is a rational prime.

Using Theorem~\ref{divisor}, one can easily adapt the proof of the the
unique factorization theorem for the Hurwitz integers given in
\cite{Conway_Smith} to any order which is a PID, yielding the
following result.

\begin{theorem}[Factorization Theorem] \label{Fact_thm} If an order
  $H$ is a (left or right) PID, then to each factorization of the norm
  $q$ of a primitive quaternion $\gamma \in H$ into a product
  $p_0p_1 \dots p_{k-1}p_k$ of rational primes, there is a
  factorization
  \[ \gamma = \pi_0 \pi_1 \dots \pi_{k-1} \pi_k \] of $\gamma$ into a
  product of prime quaternion modeled on that factorization of $q$,
  that is, with $\N(\pi_i) = p_i$.  Moreover, if
  $\gamma = \pi_0 \pi_1 \dots \pi_{k-1}\pi_k$ is any one factorization
  modeled on $p_0p_1 \dots p_{k-1}p_k$, then all the others have the
  form
  \[\gamma = \pi_0 u_1 \cdot u^{-1}_1 \pi_1 u_2 \cdot u^{-1}_2 \pi_2
    u_3 \dots u^{-1}_{k-1} \pi_{k-1} u_k \cdot u^{-1}_k \pi_k,\]
  where, $u_i$ is a unit for all $i = 1, \dots, k$, which means that
  the factorization on a given model is unique up to unit-migration.
\end{theorem}
%%%%%%%%%%%%%%%%%%%%%%%%%%%%%%%%%%%%%%%%%%%%%%%%%%%%%%%
\section{The Dedekind-Hasse's Criterion}\label{criterion}

Let $H$ be an order in the rational quaternion algebra $\A$, as in the
previous section. The usual \textit{norm} map $\N:\A\to \QQ_0^+$,
which sends $\alpha = x_0 + x_1 i + x_2 j + x_3 k$ to
$\alpha \overline{\alpha} = x_0^2 + a x_1^2 + b x_2^2 + ab x_3^2$,
plays a central role in studying the arithmetic of $H$.  The following
criterion for an order to be a principal ideal domain is based on the
criterion attributed to Dedekind and Hasse in \cite[Thm.~9.5,
p.~124]{Pollard_Diamond}.

\begin{theorem}\label{PID_criterion}
  The order $H$ is a left PID \textbf{if and only if}
  \begin{equation}\label{eq:criterion}
    \forall \rho \in \A \setminus H \; \exists \, \alpha, \beta \in H
    \colon 0<\N(\alpha \rho - \beta ) <1.
  \end{equation}
	
\end{theorem}

\begin{proof}
  Suppose $H$ is a left PID, and let $\rho \in \A \setminus H$ be
  given. As $H$ is an order, we have $H \otimes_\ZZ \QQ = \A$, thus
  there exists $n \in \NN$ such that $n \rho \in H$, with $n>1$. Let
  $\gamma$ be a generator of the ideal $Hn + H(n \rho)$. Then
  $\gamma = \alpha (n \rho) - \beta n$ for some $\alpha, \beta \in
  H$. From the fact that $ n \in H\gamma $, we have
  $n^2 = \N (n) = q \N(\gamma)$ for some $q \in \NN$. Now if
  $\N(\gamma)=n^2$, we would have that $n$ and $\gamma$ are
  associates, and as $n\rho \in H\gamma= H n$, we would get
  $\rho \in H$, which is absurd. So $0 < \N(\gamma) < \N(n)$, which
  upon dividing by $\N(n)$ yields $ 0 < \N(\alpha \rho - \beta) < 1$.
	
  Conversely, assume the right hand condition holds, and let $I$ be a
  non-zero left ideal of $H$. Let $ \gamma$ be an element of the ideal
  $I$ of minimal norm, that is, be such that
  $\N(\gamma) = \min\{\N(x): x \in I \setminus \{0\}\}$. Suppose that
  $I \not=H\gamma$. Then there exists an element
  $\delta \in I \setminus H \gamma $, so
  $\delta \gamma^{-1} \in \A \setminus H$, thus, by the hypothesis
  made, there are two elements $\alpha, \beta \in H $ such that
  $0<\N( \alpha (\delta \gamma^{-1}) - \beta ) <1$. This means that
  the element
  $\eta :=\alpha \delta - \beta \gamma \in I \setminus \{0\}$
  satisfies $\N(\eta)<\N(\gamma)$, which yields the desired
  contradiction.
\end{proof} 

\noindent\textit{Remark\/}: Analogously, one shows that $H$ is a right
PID \textbf{if and only if}
$ \forall \rho \in \A \setminus H \; \exists \, \alpha, \beta \in H
\colon 0<\N( \rho \alpha - \beta ) <1$. As mentioned above, the fact
that $H$ has an involution immediately implies that it is a left PID
if and only if it is a right PID. This fact allow us to concentrate
our attention only on left PIDs.

\vspace{3mm}

\noindent\textit{Remark\/}: Clearly, $H$ is
norm Euclidean if and only if for all $\rho \in \A$ there exists a
$\beta \in H$ such that $\N (\rho - \beta) < 1$, which is just a
particular case of condition \eqref{eq:criterion}, and we get at once
the well-known fact that a norm Euclidean order is necessarily a PID.

\vspace{3mm}

The following known result also easily follows from the
Dedekind-Hasse's criterion.

\begin{proposition}
  If an order $H$ in $\A$ is a PID, then $H$ is maximal.
\end{proposition}

\begin{proof}
  Suppose that $H$ is not a maximal order, that is, there exists an
  order $L$ such that $H \subsetneq L \subset \A$. As $L$ is an order,
  by \cite[Thm.~10.1, p.~125]{Reiner}, one has that
  $\N(L) : = \{ \N(\alpha) \colon \alpha \in L\} \subseteq \ZZ$.
	
  Pick $\ell \in L \setminus H$. Then, as $\ell \in \A \setminus H $,
  there exist elements $\alpha, \beta \in H$ such that
  $0 < \N(\alpha \ell - \beta) <1$. As
  $\alpha, \beta \in H \subsetneq L$, $\alpha \ell - \beta \in L$, and
  so $\N(\alpha \ell - \beta) \in \ZZ$, which is a contradiction.
\end{proof}
%%%%%%%%%%%%%%%%%%%%%%%%%%%%%%%%%%%%%%%%%%%%%%%%%%%%%%%
\section{Reduction to a Finite Algorithm}\label{fin_alg}

Given any rational quaternion algebra $\A$, and an order $H\subset\A$,
we define
\[\DD_\A:=\{\rho \in \A\>|\> 0<\N(\rho)<1\}\]
and
\[\SC_H:=\{\rho \in \A \setminus H \>|\> \exists \alpha,
\beta \in H:\; \alpha \rho - \beta \in \DD_\A\}.\]
It follows from Theorem~\ref{PID_criterion} that $H$ is a PID if and
only if $\SC_H = \A \setminus H$.

If we now wish to apply the Dedekind-Hasse's criterion to the order
$H$, we need to verify if every element of $\A \setminus H$ is in
$\SC_H$. The aim of this section is to provide a finite algorithm for
doing it, by showing that we only need to verify that for finitely
many elements, and for each of these elements there is a finite
procedure for checking whether or not it is in $\SC_H$.

\begin{proposition}\label{sum_prod}
  Let $\rho \in \A \setminus H$ and $\delta \in H$. Then
  \begin{enumerate}
  \item $\rho \in \SC_H$ if and only if $\rho + \delta \in \SC_H$,
  \item If $\delta \rho \in \SC_H$, then $ \rho \in \SC_H$.
  \end{enumerate}
\end{proposition}

\begin{proof}
  Suppose that $\rho \in \SC_H$. Then there are $\alpha, \beta \in H$
  such that $ 0 < \N(\alpha \rho - \beta) < 1$. As
  $\alpha\delta + \beta \in H$, and
  \[ \N \big (\alpha (\rho + \delta) - (\alpha\delta + \beta) \big ) =
    \N(\alpha \rho - \beta), \] we get that $\rho + \delta \in
  \SC_H$. The converse is now immediate.
	
  For the second property, suppose that $\delta \rho \in \SC_H$. Then,
  we have $\alpha, \beta \in H$ such that
  \[ 0 < \N(\alpha \delta \rho - \beta) < 1,\] but as
  $\alpha\delta \in H$, it immediately follows that $\rho \in \SC_H$.
\end{proof}

\begin{corollary}
  For $\rho_1, \rho_2 \in \A \setminus H$ one has
  \[ \rho_1 \equiv \rho_2 \pmod{H} \Longrightarrow (\rho_1 \in \SC_H
    \iff \rho_2 \in \SC_H).\]
\end{corollary}

Thus, in order to check if a quaternion order is a PID, it is enough
to select one element from each non-zero class in $\A/H$, and then
verify if all such elements belong to $\SC_H$. Fixing a $\ZZ$-basis
for $H$, which is then a $\QQ$-basis for $\A$, it follows that we just
need to check if all non-zero elements of $\A$ whose coefficients, in
that basis, have absolute value less or equal to $\frac{1}{2}$ are in
$\SC_H$.

Moreover, we can further reduce the number of elements of
$\A \setminus H$ we need to check as follows. As $H$ is an order, any
element of $\A$ is of the form $\frac{1}{n} \delta$, for some
$n \in \NN$ and $\delta \in H$. Thus, by Proposition~\ref{sum_prod},
if $\frac{1}{p} \delta \in \SC_H$, for a rational prime $p$, then
$\frac{1}{pq} \delta \in \SC_H$, for any $q \in \ZZ$. Thus, we need
only to check if $\frac{1}{p} \delta \in \SC_H$ for all primes $p$ and
$\delta \in H$. Now, if $\{1,v_1,v_2,v_3\}$ is a $\ZZ$-basis for $H$,
it follows that one only needs to verify if
\[\frac{d_0}{p} + \frac{d_1}{p} v_1 + \frac{d_2}{p} v_2 +
  \frac{d_3}{p} v_3\in\SC_H\] for all positive rational primes $p$,
and for $d_i$ with $d_0\geq 0$ and $\abs{d_i} \leq \frac{p}{2}$, a
finite set with no more than $\frac12 p^4$ elements.

Furthermore, since $\frac{\N(\delta)}\delta = \overline{ \delta}\in H$
for any $\delta\in H$, we have that
\[ p \nmid \N(\delta)\implies \exists_{r\in\ZZ, 0<\abs{r}< p}:
  \overline{\delta} \cdot \frac{\delta}{p} = \frac{\N(\delta)}{p}
  \equiv \frac{r} p \pmod{ H},\] and thus, for any $\delta \in H$,
$p \nmid\N(\delta) \implies \frac{\delta}{p} \in\SC_H$
{\footnotesize(since $\N(\frac{r}{p})<1$)}.

All of the above proves the following:

\begin{theorem}\label{Reduction_deltas}
  With the above notations, $H$ is a PID \; {\bf iff} \;
  $\frac{\delta}{p} \in \SC_H$, for all positive rational primes $p$,
  and for all $\delta = d_0 + d_1 v_1 + d_2 v_2 + d_3 v_3$ such that
  $\abs{d_i} \leq \frac{p}{2}$ and $ p\mid \N(\delta)$.
\end{theorem}    

Next, we wish to bound the number of primes one needs to check. Let,
as before, $H=\ZZ[1,v_1,v_2,v_3]$ be an order in the algebra
$\A = \left ( \frac{-a,-b}{\QQ} \right)$ where $a,b >0$.  Define
$d=(1+a)(1+b)$ and
\[M= \max_{\substack{1 \leq r \leq 3 \\ 0 \leq s \leq 3}}
  \abs{a_{r,s}}, \text{ where } v_r= a_{r,0} + a_{r,1} i + a_{r,2} j +
  a_{r,3} k.\]

Set $\delta = d_0 + d_1 v_1 + d_2 v_2 + d_3 v_3$ with
$\abs{d_i} \leq \frac{p}{2}$. We are now going to show that, for $p$
sufficiently large, $\frac{\delta}{p} \in \SC_H$ trivially. More
precisely, we are going to show that there is some $\alpha \in H$ such
that $\N \left(\alpha\, \frac{\delta}{p} \right)<1$.

First, we show that if we had $\abs{d_i} < \frac{p}{Q}$ for $i=1,2,3$,
and for some appropriately chosen $Q\in\NN$, then
\small
\begin{alignat*}{3}
  \N \left(\frac{\delta}{p} \right) = & \left(\frac{d_0}{p} +
    \frac{d_1}{p} a_{1,0} + \frac{d_2}{p} a_{2,0} + \frac{d_3}{p}
    a_{3,0} \right )^2 + a \left(\frac{d_1}{p} a_{1,1} +
    \frac{d_2}{p} a_{2,1} + \frac{d_3}{p} a_{3,1} \right)^2 + \\
  & + b \left( \frac{d_1}{p} a_{1,2} + \frac{d_2}{p} a_{2,2} +
    \frac{d_3}{p} a_{3,2} \right)^2 + ab \left( \frac{d_1}{p} a_{1,3}
    + \frac{d_2}{p} a_{2,3}+ \frac{d_3}{p} a_{3,3} \right)^2  \\
  < & \left( \frac{1}{2} + \frac{1}{Q} |a_{1,0}| + \frac{1}{Q}
    |a_{2,0}| + \frac{1}{Q} |a_{3,0}|\right )^2 + a \left( \frac{1}{Q}
    |a_{1,1}| + \frac{1}{Q} |a_{2,1}| + \frac{1}{Q}
    |a_{3,1}|\right )^2 + \\
  & + b \left( \frac{1}{Q} |a_{1,2}| + \frac{1}{Q} |a_{2,2}| +
    \frac{1}{Q} |a_{3,2}|\right )^2 + ab \left(\frac{1}{Q} |a_{1,3}| +
    \frac{1}{Q} |a_{2,3}| + \frac{1}{Q} |a_{3,3}|\right )^2 \\
  < & \left( \frac{1}{2} + \frac{3M}{Q}\right )^2 + a \left
    (\frac{3M}{Q}\right )^2 + b \left( \frac{3M}{Q}\right )^2 +
  ab\left( \frac{3M}{Q}\right )^2 \\
  & = \frac{1}{4} + 3\, \frac{M}{Q} + 9d \left(\frac{M}{Q}\right)^2
\end{alignat*}
\normalsize Since
\[\frac{1}{4} + 3\, \frac{M}{Q} + 9d
\left(\frac{M}{Q}\right)^2 \leq 1 \iff Q \geq 2M (1+\sqrt{1+3d}),
\]
one gets the following result:
\begin{lemma}\label{estimateN}
  With the above notations, setting
  $Q = \left \lceil 2 M (1+\sqrt{1+3d}) \right \rceil$, one has that
  if $\abs{d_i} < \frac{p}{Q}$ for $i=1,2,3$ then
  $ \N\left(\frac{\delta}{p}\right) < 1$.
\end{lemma}

Now, we only need to guarantee that we can reduce to the case in
which, for $i=1,2,3$, one has $\abs{d_i} < \frac{p}{Q}$. For this
purpose, we use Theorem~201 in \cite[Chap.~IX, p.~170]{Hardy_Wright},
with a slightly modified statement to better suit our notations.

\begin{theorem}\label{Hardy-Wright-Thm}
  Given $\xi_1, \dots, \xi_n \in \RR$ and any $Q \in \RR^+$, we can
  find $u \in \{1, 2, \dots, Q^n\}$ so that, for all $i$, $u \xi_i$
  differs from an integer by less than $\frac{1}{Q}$. In other words,
  \[ \abs{ u \xi_i - k_i} < \frac{1}{Q} \] for some $k_i \in \ZZ$, and
  for all $i$.
\end{theorem}

Let $\xi_i = \frac{d_i}{p}$, for some $p$ prime and
$\abs{d_i} \leq \frac{p}{2}$, and let $Q$ as in
Lemma~\ref{estimateN}. Then this theorem implies that there is
$u \in \{ 1,2, \dots, Q^3\}$, and there are $k_i\in\ZZ$, for
$i=1,2,3$, so that $\abs{u \frac{d_i}{p} - k_i} <
\frac{1}{Q}$. Setting $r_i = u d_i - p k_i\in\ZZ$, one has
$u d_i \equiv r_i \pmod{p}$ and $\abs{\frac{r_i}{p}} < \frac{1}{Q}$.

By Proposition~\ref{sum_prod}, in order to show that
$\frac{\delta}{p}\in\SC_H$, it is enough to show that
$u\, \frac{\delta}{p}\in\SC_H$.  We have \small
\[ u\, \frac{\delta}{p} = \frac{u d_0}{p} + \frac{u d_1}{p} v_1 +
  \frac{u d_2}{p} v_2 + \frac{u d_3}{p} v_3 \equiv \frac{d_0'}{p} +
  \frac{r_1}{p} v_1 + \frac{r_2}{p} v_2 + \frac{r_3}{p} v_3
  \pmod{H}, \]
\normalsize
where $\abs{r_i} < \frac{p}{Q}$ and $d_0' \equiv ud_0 \pmod{p}$ is
such that $\abs{d_0'} < \frac{p}{2}$. Let $\gamma$ be the quaternion
on the right hand side. Then, by Lemma~\ref{estimateN}, we have
$\N \left(\gamma \right) < 1 $. Since $1\leq u\leq Q^3$ one sees that
for $p > Q^3$ one has that $\gamma \notin H$, so that
$\gamma\in\SC_H$. That is, we have proved the following result.

\begin{theorem}\label{Reduction_primes}
  Let $H=\ZZ[1,v_1,v_2,v_3]$ be an order in the rational quaternion
  algebra $\A = \left( \frac{-a,-b}{\QQ}\right)$, where $a,b \in
  \NN$. Let $d=(1+a)(1+b)$, and
  \[M= \max_{\substack{1 \leq r \leq 3 \\ 0 \leq s \leq 3}}
    \abs{a_{r,s}}, \text{ where } v_r= a_{r,0} + a_{r,1} i + a_{r,2} j
    + a_{r,3} k.\]
  Then, to check if $H$ is a PID, we only need to check if
  \[\frac{\delta}{p} = \frac{d_0}{p} + \frac{d_1}{p} v_1 +
    \frac{d_2}{p} v_2 + \frac{d_3}{p} v_3 \in \SC_H \] for
  $\delta \in H$ with $d_0\geq 0$, $\abs{d_i} \leq \frac{p}{2}$, and
  such that $ p \mid \N(\delta)$, and for
  \[p \text{ prime} \text{ with } p < \left \lceil 2 M (1+\sqrt{1+3d})
    \right \rceil^3.\]
\end{theorem}

\begin{remark}
  In this last proof, we have only used Theorem~\ref{Hardy-Wright-Thm}
  for the three last coefficients of $\delta$ because, although
  bounding all four coefficients would lead to a smaller bound in
  Lemma~\ref{estimateN}, that would entail a fourth power in the last
  result. It turns out that, when working with a particular case, one
  can get a smaller upper bound for the primes $p$ either by bounding
  the coefficients of $\delta$ through tighter estimates than the
  value $M$ above, or by using Theorem~\ref{Hardy-Wright-Thm} for only
  two coefficients, which can be done when the algebra one works with
  is of the form $\left( \frac{-1,-b}{\QQ} \right)$.
\end{remark}

We now have a criterion to check if a given order is a PID by
verifying if only a finite number of quaternions
$\rho \in \A \setminus H$ are in $\SC_H$. However, checking if a given
element $\rho$ is in this last set is still a problem, as there exists
an infinite number of possibilities for $\alpha, \beta$ to consider to
check whether $0 <\N(\alpha \rho - \beta) <1$. Let us now see how that
this can also be reduced to checking a finite number of possibilities.

\begin{proposition}
  Let $H$ be an order in a quaternion algebra $\A$. Let $p$ be a
  rational prime and $\rho = \frac{\delta}{p}$ in $\A \setminus
  H$. For any $\alpha \in H$ we have the following equivalence:
  \[ \exists \beta\in H \colon \alpha \rho - \beta \in \DD_\A \iff
    \forall \alpha_1 \in [\alpha]\; \exists \beta_1\in H \colon
    \alpha_1 \rho - \beta_1 \in \DD_\A, \]
  where $[\alpha]$ denotes the equivalence class of $\alpha$ in
  $H/pH$.
\end{proposition}

\begin{proof}
  Let us fix a prime $p$, and set
  $\rho = \frac{\delta}{p} \in \A \setminus H$.
	
  Let $\alpha, \beta$ be elements of $H$. Let $\alpha_1 \in H$ be any
  representative of $[\alpha] \in H/pH$, and let $\gamma \in H$ be
  such that $\alpha = \alpha_1 + p \gamma$. Then
  \[\alpha\, \frac{\delta}{p} - \beta = (\alpha_1 + p \gamma)\,
    \frac{\delta}{p} - \beta = \alpha_1\, \frac{\delta}{p} - (\beta -
    \gamma \delta). \]
  Thus, if $\alpha, \beta$ are such that
  $\alpha \frac{\delta}{p} - \beta \in \DD_\A$, then taking
  $\beta_1 = (\beta - \gamma \delta)$ one also has
  $\alpha_1 \frac{\delta}{p} - \beta_1 \in \DD_\A$.
\end{proof}

Therefore, if we put the previous Proposition together with Theorem
\ref{Reduction_primes}, we get the following result.

\begin{theorem}\label{Reduction_alphas}
  To check if an order $H$ in an algebra $\A$ is a PID, we only need
  to check if, for all $\rho= \frac{\delta}{p}$ with $p$ prime such
  that $p < \left \lceil 2 M (1+\sqrt{1+3d}) \right \rceil ^3$ and
  with $\delta \in H$ such that $d_0\geq 0$,
  $\abs{d_i} \leq \frac{p}{2}$, and such that $p \mid \N(\delta)$,
  there exists $\alpha, \beta \in H$, where
  $\alpha = a_0 + a_1 v_1 + a_2 v_2 + a_3 v_3$ with
  $\abs{a_i} \leq \frac{p}{2}$, such that
  $\alpha \rho - \beta \in \DD_\A$.
\end{theorem}

All that is left to consider is the best way to find if there is a
suitable $\beta$ for each pair $\left( \rho, \alpha \right)$ in the
previous theorem. In other words, we have to solve
\begin{equation}\label{min_problem}
  \min_{\beta \in H} \N(\gamma - \beta),
\end{equation}
where $\gamma$ runs over the possible pairs $(\rho,\alpha)$. To put it
another way, we want to know if the closest quaternion to
$\alpha \rho$ that belongs to the order is at a distance less than 1
or not. This is related to the closest vector problem, and it can be
solved using an algorithm presented by László Babai in \cite{Babai}.

Let us first identify the algebra $\A$ with $\RR^4$ as follows
\begin{alignat*}{3}
  \phi \colon \A = \left(\frac{-a,-b}{\QQ} \right) & \longrightarrow
  \ \RR^4 \\
  t + x i + y j + x k \ & \longmapsto \
  (t,x\sqrt{a},y\sqrt{b},z\sqrt{ab}).
\end{alignat*}
One can easily see that the bilinear form associated with the algebra
$\A$, that is,
$B(\alpha, \beta) =\frac{1}{2} \Tr(\alpha\overline{\beta})$, is equal
to the inner product in $\RR^4$, that is,
$\phi(\alpha) \cdot \phi(\beta)$. Thus, problem of finding the minimun
in \eqref{min_problem} is equivalent to finding
\[ \min_{v \in \phi(H)} \norm{u-v}, \]
where $u=\phi(\gamma)$. If $\gamma = \sum_{i=0}^3 u_i v_i$, then
$u= \sum_{i=0}^3 u_i \phi(v_i)$, and Babai's Algorithm works by
rounding each $u_i$ to the nearest integer.

Now, Theorem~6.34 of \cite[p.~380]{Cript} tells us that this algorithm
finds an approximate solution to the closest vector problem if the
vectors in the basis are sufficiently orthogonal to one another, which
is to say that its \textit{Hadamard Ratio} (cf.~\cite[p.~373]{Cript})
is reasonably close to 1. This ratio is defined for any given basis,
$(v_0, \dots v_n)$, of a lattice in $\RR^n$ by
\[ \hh(v_0,\dots,v_{n-1}) =
\left(\frac{\abs{\det(A)}}{\prod_{i=0}^{n-1}\norm{v_i}}
\right)^{1/n},\] where $A$ is the matrix whose columns are the
coordinates of $v_i$.

Although this ratio might not be close to one for most bases, in our
case the algorithm did succeeded in find a suitable element, as we
will see in the next section. For this reason, the algorithm failing
to find a pair $(\alpha, \beta)$ for a particular $\frac{\delta}{p}$
does not imply immediately that the order is not a PID. But we can
then try to deal with each of those exceptions individually to find
whether the criterion~\eqref{eq:criterion} of
Theorem~\ref{PID_criterion} fails or not. However, this did not happen
for the orders we are concerned about in this paper.

%%%%%%%%%%%%%%%%%%%%%%%%%%%%%%%%%%%%%%%%%%%%%%%%%%%%%%%
\section{The Maximal Orders of Discriminant 7 and
  13}\label{orders7and13}

In this section, we will apply the Dedekind-Hasse's Criterion,
according to Theorem~\ref{Reduction_alphas}, to the maximal orders of
discriminant 7 and 13, in order to give an arithmetic proof that they
are both PIDs.

\subsection{Maximal Order of Discriminant 7}

Consider the rational quaternion algebra
$\A = \left( \frac{-1,-7}{\QQ} \right)$, and its maximal order of
discriminant 7, $\hh_{1,7} ;= \ZZ[1, v_1,v_2,v_3]$ where
\[ v_1= i, \quad v_2 = \frac{1}{2} (1+j), \quad v_3 =
  \frac{1}{2}(i+k),\] with norm form
\small
\[ \N( d_0 + d_1 v_1 + d_2 v_2 + d_3 v_3 ) = \left( d_0 +
    \frac{d_2}{2} \right)^2 + \left( d_1 + \frac{d_3}{2} \right)^2 + 7
  \left( \frac{d_2}{2} \right)^2 + 7 \left( \frac{d_3}{2} \right)^2.\]
\normalsize

According to Theorem~\ref{Reduction_primes}, we only need to check if
elements of the form
$\frac{\delta}{p}= \frac{d_0}{p} + \frac{d_1}{p} v_1 + \frac{d_2}{p}
v_2 + \frac{d_3}{p} v_3$ are in $ \SC_H$, for each of the primes $p$
such that $ p <\lceil 2 M(1+\sqrt{1+3d}) \rceil^3$, and for
$\delta \in H$ with $d_0\geq 0$, $\abs{d_i} \leq \frac{p}{2}$ and such
that $p \mid \N(\delta)$.  We easily see that, for this order, $d=16$
and $M=1$, and thus it would be enough to deal with the primes
$p < 16^3$.

Actually, this upper bound can be improved, as we are working with a
concrete order, and because our algebra is of the form
$\left( \frac{-1,-b}{\QQ} \right)$ we can bound only the coefficients
of $v_2$ and $v_3$ in order to get a better bound for $p$. In order to
see that, let us then reproduce, for this particular case, the
calculations done in the proof of Lemma~\ref{estimateN} to find the
smallest $Q$ such that $\abs{d_i} < \frac{p}Q$, for $i=2,3$, implies
$\N \left( \frac{\delta}{p} \right) <1$:

\begin{alignat*}{3}
  \N \left(\frac{\delta}{p} \right) & = \left( \frac{d_0}{p} +
    \frac{d_2}{2p} \right)^2 + \left( \frac{d_1}{p} + \frac{d_3}{2p}
  \right)^2 + 7 \left( \frac{d_2}{2p} \right)^2 + 7
  \left(  \frac{d_3}{2p} \right)^2 \\
  & < \left( \frac{1}{2} + \frac{1}{2Q} \right )^2 + \left(
    \frac{1}{2} + \frac{1}{2Q} \right )^2 + 7 \left( \frac{1}{2Q}
  \right )^2 + 7 \left( \frac{1}{2Q} \right )^2 \\
  & = \, \frac{1}{2} + \frac{1}{Q} + \frac{4}{Q^2}.
\end{alignat*}
Hence, one sees that, when $Q \geq 4$, one has
$\frac{\delta}{p} \in \SC_H$. Once again, using the same arguments as
before, we can see that we only need to check if
$\frac{\delta}{p} \in \SC_H$ for $\delta \in H$ as above, and for
primes $p$ with $p < 4^2$, which is a better upper bound than
before. Therefore we just need to verify if $\frac{\delta}{p}$ is in
$\SC_H$ for the primes $p = 2,3, 5, 7, 11, 13$, and $\delta$ with
$d_0\geq 0$, $\abs{d_i} \leq \frac{p}{2}$ and such that
$p \mid \N(\delta)$.

Using the PARI/GP \cite{PARI} code presented in the next section, we
were able to find how many quaternions of the form $\frac{\delta}{p}$
need to be checked according to Theorem~\ref{Reduction_primes}. Note
that the Hadamard ratio of this basis is
\[\hh(v_0,v_1,v_2,v_3) = \left( \frac{7}{8}\right)^{\frac{1}{4}}
  \approx 0.967,\]
which suggest that the Babai's algorithm will work. The code was run
in a laptop of 16GB of RAM, and it took less than a second to check
102 quaternions over six primes.

Finally, using Theorem~\ref{Reduction_alphas} and the corresponding
PARI/GP functions presented in Section~\ref{sec:PARI} below, we were
able to find, for each quaternion of the above form, pairs
$\alpha, \beta$ such that
$\alpha \frac{\delta}{p} - \beta \in \DD_\A$, thus proving the
following:

\begin{theorem}
  $\hh_{1,7}$ is a PID.
\end{theorem}
%%%%%%%%%%%%%%%%%%%%%%%%%%%%%%%%%%%%%%%%%%%%%%%%%%%%%%%
\subsection{Maximal Order of Discriminant 13}

Consider now the rational algebra
$\A = \left( \frac{-7,-13}{\QQ} \right)$, with its maximal order of
discriminant 13, $\hh_{7,13} := \ZZ[1, v_1,v_2,v_3]$ where
\[ v_1= \frac{1}{2}(1+i), \quad v_2 = \frac{1}{7}(i+k), \quad v_3 =
  \frac{1}{14} (7 + i + 7 j+ k),\] and norm form
\begin{multline*}
  \N ( d_0 + d_1 v_1 + d_2 v_2 + d_3 v_3 ) = \left( d_0 +
    \frac{d_1}{2} + \frac{d_3}{2}\right)^2 + 7 \left( \frac{d_1}{2} +
    \frac{d_2}{7} + \frac{d_2}{14} \right)^2 +\\+ 13 \left(
    \frac{d_3}{2} \right)^2 + 91\left( \frac{d_2}{7} + \frac{d_3}{14}
  \right)^2.
\end{multline*}

Once again, by Theorem~\ref{Reduction_primes}, it is enough to check
if elements of the form
$\frac{\delta}{p}= \frac{d_0}{p} + \frac{d_1}{p} v_1 + \frac{d_2}{p}
v_2 + \frac{d_3}{p} v_3$ are in $ \SC_H$, for $\delta \in H$ with
$\abs{d_i} \leq \frac{p}{2}$ and such that $p \mid \N(\delta)$, and
for primes $p <39^3$, as for this order we have $d=112$ and
$M=\frac{1}{2}$. As above, this upper bound can be improved by setting
$\abs{d_i} < \frac{p}Q$, with $Q$ to be chosen as the smallest
positive integer so that $\N \left(\frac{\delta}{p} \right) <1$, as
follows:
\begin{alignat*}{3}
  \N \left(\frac{\delta}{p} \right) & = \left( \frac{d_0}{p} +
    \frac{d_1}{2p} + \frac{d_3}{2p}\right)^2 + 7 \left( \frac{d_1}{2p}
    + \frac{d_2}{7p} + \frac{d_2}{14p} \right)^2 + 13
  \left( \frac{d_3}{2p} \right)^2 \\
  & \quad \ + 91\left(  \frac{d_2}{7p} + \frac{d_3}{14p} \right)^2 \\\
  & < \left( \frac{1}{2} + \frac{1}{2Q} + \frac{1}{2Q}\right)^2 + 7
  \left( \frac{1}{2Q} + \frac{1}{7Q} + \frac{1}{14Q} \right)^2 + 13
  \left( \frac{1}{2Q} \right)^2 + \\
  & \quad \ + 91\left(  \frac{1}{7Q} + \frac{1}{14Q} \right)^2 \\
  & = \, \frac{1}{4} + \frac{1}{Q} + \frac{12}{Q^2}.
\end{alignat*}
It then easily follows that $\frac{\delta}{p} \in \SC_H$ for
$p > \lceil 4.72 \rceil^3 = 125$. Thus, we can use the PARI/GP
functions presented in the next section to verify if the remaining
elements of the form $\frac{\delta}{p}$ are in $\SC_H$, that is, for
$\delta$ such that $p \mid \N(\delta)$ and with $d_0\geq 0$,
$\abs{d_i} \leq \frac{p}{2}$ and such that $p \mid \N(\delta)$, and
for the primes $p$ up to $125$. As the Hadamard ratio of this basis is
\[\hh(v_0,v_1,v_2,v_3)= \left(\frac{13}{16} \right)^{\frac{1}{4}}
\approx 0.949,\] the Babai's algorithm is expected to work.  The
code was once again run in a laptop of 16GB of RAM, and it took 45
minutes to check a total of $1\,339\,411$ quaternions over the first
30 primes, from 2 to 113, hence proving the following:

\begin{theorem}
  $\hh_{7,13}$ is a PID.
\end{theorem}

%%%%%%%%%%%%%%%%%%%%%%%%%%%%%%%%%%%%%%%%%%%%%%%%%%%%%%%
\section{PARI/GP Code}\label{sec:PARI}

The first step in order to code the algorithm is finding a way to
represent an arbitrary element of a given quaternion order. In order
to do this, we first define the function \textbf{quatH}, that
represents the quaternion $\alpha = a + b v_1 + c v_2 + d v_3$ in the
quaternion order $H=\ZZ[1,v_1,v_2, v_3]$ as a matrix, using the linear
operator of right multiplication:
\begin{alignat*}{3}
  & H \ & \longrightarrow & \; H \\
  & h \ & \longmapsto & \; \alpha h
\end{alignat*}
which, for the orders $\hh_{1,7}$ and $\hh_{7,13}$, yields the
matrices
\small
\[ B_7=
  \begin{pmatrix}
    a & -b-d & -2c & -2d \\
    b & a+c & -2d & 2c \\
    c & d & a+c & -b \\
    d & -c & b+d & a \\
  \end{pmatrix}, \; B_{13}=\begin{pmatrix}
    a & -2b-c-d & -2c-2d & -4d \\
    b & a+b+d & 2d & -2c \\
    c & -2d & a+b & 2b+c \\
    d & c+d & -b & a+d \\
  \end{pmatrix}, \]
\normalsize respectively.

Next, we define the functions \textbf{qHconj} and \textbf{qHnorm}
that, given a quaternion $q$ in matrix form, return the conjugate and
the norm of $q$, respectively. We also define the function
\textbf{quatb}, that given a quaternion, in matrix form, returns the
coefficients in vector form for better reading.

Here, we only present the code for dealing with $\hh_{1,7}$, as the
one for dealing with $\hh_{7,13}$ is entirely analogous.

\footnotesize

\begin{lstlisting}
/* Representation of a H_7-quaternion as a matrix and
its conjugate */

quatH(a,b,c,d)=[a,-b-d,-2*c,-2*d;b,a+c,-2*d,2*c;c,d,a+c,-b;
                d,-c,b+d,a];

qHconj(a)=quatH(a[1,1]+a[3,1],-a[2,1],-a[3,1],-a[4,1]);

qHnorm(q)=(q*qHconj(q))[1,1];

/*Given any quaternion in matrix form, returns it in the form 
[a,b,c,d] for any basis */

quatb(q)=[q[1,1],q[2,1],q[3,1],q[4,1]]; 
\end{lstlisting}

\normalsize

Using this, we define the function \textbf{PIDChecking} that takes as
input the bound for the primes that one needs to check to see if the
order is a PID, for example the bound given in
Lemma~\ref{estimateN}. Using the previous notations, this function
will go through all $\delta$ and $p$ in the conditions of Theorem
\ref{Reduction_primes}, and it prints the quadruples
$[\delta, p, \alpha, \beta]$, whenever it finds $\alpha, \beta$ such
that $0 < \N(\alpha \frac{\delta}{p} - \beta) <1$. If it does not find
any such elements, it will add the pair $(\delta,p)$ to a list, which
it returns when the verification is finished.

The function \textbf{PIDChecking} uses the function
\textbf{alpha\_beta} that takes as input a pair $(\delta,p)$, and runs
through all possible $\alpha$ in the conditions of
Theorem~\ref{Reduction_alphas}, and then calls the function
\textbf{beta} that yields the quaternion $\beta$ obtained by rounding
the coefficient of $\alpha \frac{\delta}{p}$ to the nearest integer,
but away from zero when the value is half of an odd integer. If it
succeeds, it will return the pair $(\alpha, \beta)$, and if it fails,
it will simply return the empty list.

\footnotesize
\begin{lstlisting}
/* The next function checks, first, which rho=delta/p has N(rho)>1
and such that p | N(delta), for p < bp; prints how many quaternions
need to be checked. Next, it uses the function alpha_beta to get a
pair (alpha, beta) such that 0 < N(alpha * rho - beta) < 1, and prints
the empty list if such pair could not be found. Returns List P of
quaternions  delta such that no pair (alpha, beta) was found */
	
PIDChecking(bp)= {
    local(N, D, count, Rho, AUX, P);
    P=List([]);
    Rho=List([]);
    count=0;
	
    forvec( X=[[0,1],[0,1],[0,1],[0,1]], 
        N = qHnorm(quatH(X[1],X[2],X[3],X[4])); 
        if(N >= 2^2 && N%2 ==0, 
            count++;  
            listput(~Rho,[quatH(X[1],X[2],X[3],X[4]),2])));
	
    print("Needs to be checked: ",count , "for p=",2); 
	
    forprime(p=3, bp, 
        D = (p-1)/2; 
        count=0; 
        forvec(X = [[0,D],[-D,D],[-D,D],[-D,D]],  
        if(X[1]==0 && (X[2]<0 || (X[2]==0 && (X[3]<0 || 
                          (X[3]==0 && X[4]<=0)))), next()); 
        N = qHnorm(quatH(X[1],X[2],X[3],X[4])); 
        if(N >= p^2 && N%p ==0, 
            count++; 
            listput(~Rho,[quatH(X[1],X[2],X[3],X[4]),p]))); 
	
    print("Needs to be checked: ",count, " for p=",p)); 
    print("Total:",length(Rho));
	
    foreach(Rho, V, 
        AUX=alpha_beta(V[1],V[2]); 
        if(length(AUX)==2, 
            print([quatb(V[1]), V[2]], AUX), 
            listput(~P,V))); 
	
    return(P);
};

/* Given a delta and a prime p, the next function runs over all
alpha with coefficients between [-p/2, p/2] and checks if the
beta given by next function, beta() is such that
0 < N(alpha*delta/p - beta) < 1. Stops at the first succeful
iteration */

alpha_beta(d,p)={
    local(M, m, x, N);
    if (p==2, M=1; m=0, M=(p-1)/2; m=-M);
	
    forvec(A = [[m,M],[m,M],[m,M],[m,M]], 
        if(A==[0,0,0,0], next());  
        x = (quatH(A[1],A[2],A[3],A[4])*d)/p;
        N=qHnorm(x-beta(x)); 
        if((N < 1) && ( N > 0), 
            return([A , quatb(beta(x))])));
	
    return([]);
};

/* Given a quaternion with coefficients not necessarily integers, it
will round it away from zero if it is half an odd number */

beta(x)={
	local(b);
	b=[0,0,0,0];
	AUX=quatb(x);	
	for(i=1, 4, 
	if ( AUX[i]>=0, b[i]=round(AUX[i]), 
	b[i]=- round(- AUX[i])));
	
	return(quatH(b[1],b[2],b[3],b[4]));
};
\end{lstlisting}

\normalsize

\vskip20pt \noindent {\bf Acknowledgments.} The authors were partially
supported by CMUP, member of LASI, which is financed by national funds
through FCT --- Fundação para a Ciência e a Tecnologia, I.P., under
the project with reference UIDP/00144/2025, DOI:
\url{https://doi.org/10.54499/UID/00144/2025}. The first author was
also supported by the FCT doctoral scholarship with reference
2022.13732.BD and DOI \url{https://doi.org/10.54499/2022.13732.BD}.

%%%%%%%%%%%%%%%%%%%%%%%%%%%%%%%%%%%%%%%%%%%%%%%%%%%%%%%
 
%%%%%%%%%%%%%%%%%%%%%%%%%%%%%%%%%%%%%%%%%%%%%%%%%%%%%%%
\end{document}